\documentclass[11pt]{article}

\usepackage{amsmath,amssymb,amsthm,mathtools}
\usepackage{enumitem}
\usepackage[margin=1in]{geometry}
\usepackage{hyperref}
\usepackage[nameinlink,noabbrev]{cleveref}
\hypersetup{hidelinks}

\newtheorem{theorem}{Theorem}[section]
\newtheorem{proposition}[theorem]{Proposition}
\newtheorem{lemma}[theorem]{Lemma}
\newtheorem{corollary}[theorem]{Corollary}

\theoremstyle{definition}

\newtheorem{problem}[theorem]{Open Problem}
\newtheorem{example}[theorem]{Example}

\theoremstyle{remark}
\newtheorem{remark}[theorem]{Remark}

\theoremstyle{plain}
\newtheorem{claim}{Claim}

\crefname{problem}{Open Problem}{Open Problems}
\Crefname{problem}{Open Problem}{Open Problems}

\newcommand{\K}{\mathbb K}
\newcommand{\N}{\mathbb N}
\newcommand{\Z}{\mathbb Z}

\newcommand{\id}{\operatorname{id}}

\newcommand{\cl}[1]{\overline{#1}}

\title{A characterization of the reversibility of linear cellular automata}
\author{Alonso Castillo-Ramirez\footnote{Email: alonso.castillor@academicos.udg.mx}\\[2pt]
\small Universidad de Guadalajara, Centro Universitario de Ciencias Exactas e Ingenier\'ias,\\ 
\small Departamento de Matemáticas. Guadalajara, M\'exico.}
\date{}

\begin{document}

\maketitle

\begin{abstract}
Let $G$ be a group, let $\K$ be a field, and let $V$ be a
$\K$-vector space. We characterize when there exists a bijective linear
cellular automaton $V^G\to V^G$ whose inverse is not a cellular automaton.
More precisely, we prove that such an automaton exists if and only if $G$
is not locally finite and
$\dim_{\K}V\geq|\K|^{\aleph_0}$. In particular,
$|\K|^{\aleph_0}$ is the exact dimension threshold for
non-reversibility. This answers an open problem proposed by
T. Ceccherini-Silberstein and M. Coornaert.
\end{abstract}




\section{Introduction}

Let $G$ be a group and let $A$ be a set. A cellular automaton over
$G$ with alphabet $A$ is a map $\tau\colon A^G\to A^G$ defined by a
local rule $\mu\colon A^M\to A$, where $M$ is a finite subset of $G$. A
cellular automaton is \emph{reversible} if it is bijective and its inverse
is also a cellular automaton. If $A$ is finite, every bijective cellular
automaton is reversible by compactness and the Curtis--Hedlund--Lyndon theorem; see
\cite[Chapter~1]{CSCAG2} and the classical paper \cite{Hedlund1969}. The
same conclusion holds for an arbitrary alphabet when $G$ is locally finite
\cite[Proposition~6.1]{CSC2011}.

Suppose now that $A=V$ is a vector space over a field $\K$, and that
the local defining map $\mu\colon V^M \to V$ is linear. For any group $G$
and any finite-dimensional vector space $V$, every bijective linear
cellular automaton $\tau\colon V^G\to V^G$ is reversible
\cite[Theorem~1.1]{CSC2011}. For infinite-dimensional vector spaces,
T. Ceccherini-Silberstein and M. Coornaert proposed the following problem;
see \cite{CSC2011} and \cite[Open Problem~3]{CSCAG2}.

\begin{problem}\label{prob:op3}
Let $G$ be a periodic group that is not locally finite, and let $V$ be an infinite-dimensional vector space over a field $\K$.  Does there exist a bijective linear cellular automaton $\tau\colon V^G\to V^G$ that is not reversible?
\end{problem}

Recently, J. Yang solved the corresponding non-linear reversibility problem for
arbitrary alphabets, as well as the related closed-image problems. They proved
that a group $G$ is locally finite if and only if every bijective cellular
automaton over every alphabet is reversible. In fact, if $G$ is not
locally finite, such an example exists over every infinite alphabet
\cite[Theorem~1.1]{YangReversibility2026}. Yang also characterized local
finiteness in terms of the closed-image property, both for arbitrary
cellular automata over infinite alphabets and for linear cellular automata
over infinite-dimensional vector-space alphabets
\cite[Theorems~1.3 and~1.5]{YangClosedImage2026}. These results solve Open
Problems~2, 6, and~7 of \cite{CSCAG2}. 

The answer to \Cref{prob:op3}, however, also depends on the dimension of $V$. Our main result is the following.

\begin{theorem}\label{thm:main}
Let $G$ be a group, let $\K$ be a field, and let $V$ be a $\K$-vector space.  The following are equivalent:
\begin{enumerate}[label=\textup{(\roman*)}]
  \item There exists a bijective linear cellular automaton $\tau\colon V^G\to V^G$ that is not reversible;
  \item $G$ is not locally finite and $\dim_{\K}V\geq |\K|^{\aleph_0}$.
\end{enumerate}
\end{theorem}

By the Erd\H{o}s--Kaplansky theorem,
$\dim_{\K}(\K^{\N})=|\K|^{\aleph_0}$; see
\cite[Chapter~IX, Section~5, Theorem~2]{Jacobson1975}. Hence the cardinal
appearing in \Cref{thm:main} is the dimension of the countable direct product of copies of $\K$. This also shows that the answer to \Cref{prob:op3} is affirmative if and only if
$\dim_{\K}V\geq|\K|^{\aleph_0}$. In particular, infinite dimensionality of $V$ alone does not determine its answer. 

This paper is organized as follows. In Section~2, we recall some basic facts
about cellular automata over groups. In Section~3, we use linearly
topologized $\K$-vector spaces to prove the implication
\textup{(i)}$\Rightarrow$\textup{(ii)} of \Cref{thm:main}. We also point
out a gap in the proof of \cite[Theorem~1.2]{CSC2011}. In Section~4, we
prove the implication \textup{(ii)}$\Rightarrow$\textup{(i)} of \Cref{thm:main} by completing the linear construction of
\cite{YangClosedImage2026}.


\section{Cellular automata over groups}

Throughout, $\N=\{0,1,2,\ldots\}$. We recall some basic facts about
cellular automata over groups, following \cite{CSCAG2}. Let $G$ be a group
and let $A$ be a set. The \emph{shift action} of $G$ on $A^G$ is
defined by
\[
  (g x)(h)=x(g^{-1}h), \quad \forall g,h\in G,\ x\in A^G.
\]
A map $\tau\colon A^G\to A^G$ is a \emph{cellular automaton} if there exist a finite set $M\subseteq G$, called a \emph{memory set}, and a map $\mu\colon A^M\to A$ such that
\begin{equation}\label{eq:ca-local-rule}
  \tau(x)(g)=\mu\bigl((g^{-1}x)|_M\bigr), \quad \forall x\in A^G,\ g\in G. 
\end{equation}

By the generalized Curtis--Hedlund--Lyndon theorem
\cite[Theorem~1.9.1]{CSCAG2}, a map $\tau\colon A^G\to A^G$ is a cellular
automaton if and only if it is $G$-equivariant and uniformly continuous
with respect to the prodiscrete uniform structure on $A^G$.

Suppose that $A=V$ is a vector space over a field $\K$. A cellular
automaton $\tau\colon V^G\to V^G$ is \emph{linear} if $\tau$ is
$\K$-linear; equivalently, if it admits a linear local defining map. The
linear Curtis--Hedlund--Lyndon theorem \cite[Theorem~8.1.2]{CSCAG2} states
that a $\K$-linear map $\tau\colon V^G\to V^G$ is a cellular automaton
if and only if it is $G$-equivariant and continuous with respect to the
\emph{prodiscrete topology} on $V^G$, which is the product topology obtained from the discrete topology on $V$.

The prodiscrete topology on $V^G$ is Hausdorff and totally disconnected,
and it is compact if and only if $V$ is finite. It is metrizable whenever
$G$ is countable. Conversely, if $V$ is nonzero and $V^G$
is metrizable, then $G$ is countable \cite[Chapter~1]{CSCAG2}.

Let $\tau\colon A^G\to A^G$ be a cellular automaton with a memory set
$M$ contained in a subgroup $H$ of $G$. Its \emph{restriction} is the
cellular automaton $\tau_H\colon A^H\to A^H$; see
\cite[Section~1.7]{CSCAG2} and \cite{CSCInduction2009}. The map $\tau$ is
injective, surjective, or bijective if and only if $\tau_H$ has the
corresponding property \cite[Proposition~1.7.4]{CSCAG2}. Moreover, $\tau$
is reversible if and only if $\tau_H$ is reversible
\cite[Proposition~1.10.4]{CSCAG2}.


\section{The first implication of \texorpdfstring{\Cref{thm:main}}{the main theorem}}

In this section, $V$ denotes a topological vector space over a discrete
field $\K$. Translations and multiplication by nonzero scalars are
homeomorphisms of $V$. Hence every neighborhood of a point is the translate
of a neighborhood of $0$. We say that $V$ is \emph{linearly topologized}
if it has a neighborhood basis at $0$ consisting of vector subspaces
\cite{Positselski2024}. Every open subspace $U$ of $V$ is also closed,
because
$V\setminus U=\bigcup_{x\in V\setminus U}(x+U)$ and every coset in this
union is open.

\begin{example}\label{example}
Let $G$ be a group and let $V$ be a discrete vector space. Then $V^G$,
with the prodiscrete topology, is linearly topologized. Indeed, the
subspaces
\begin{equation}\label{eq:UF}
  U_F:=\{x\in V^G : x|_F=0\},
  \qquad F\subseteq G\text{ finite},
\end{equation}
form a neighborhood basis at $0$. If $G$ is countably infinite, fix an
enumeration $G=\{g_0,g_1,\ldots\}$ and set
\[
 d(x,y)=
 \begin{cases}
  0,&x=y,\\
  2^{-m},&m=\min\{j:x(g_j)\neq y(g_j)\}.
 \end{cases}
\]
This metric induces the prodiscrete topology and is
complete; see also \cite[Exercise~1.2]{ExercisesCAG}. In fact, a Cauchy
sequence eventually stabilizes at every coordinate, and its coordinatewise
limit is also its limit with respect to $d$. When $G$ is finite, the
prodiscrete topology is discrete, and hence complete and metrizable.
Therefore, $V^G$ is complete and metrizable whenever $G$ is countable.
\end{example}

In contrast with the previous example, for $n\geq1$ the vector space
$\mathbb{R}^n$ with its usual topology is a topological vector space, but
it is not linearly topologized.

We shall use the following standard fact; cf.\ \cite[Theorem~3.7.4]
{NariciBeckenstein2011}. We include its proof for
completeness.

\begin{lemma}\label{lem:complete-quotient}
Let $X$ be a complete metrizable linearly topologized $\K$-vector space,
and let $C\subseteq X$ be a closed subspace. Then the quotient $X/C$, with
the quotient topology, is also a complete metrizable linearly topologized
$\K$-vector space.
\end{lemma}

\begin{proof}
Let $q\colon X\to X/C$ be the quotient map, and choose a decreasing
countable neighborhood basis $(U_n)_{n\geq0}$ at $0$ consisting of open
subspaces, with $U_0=X$. Since $q$ is open, the subspaces
$B_n:=q(U_n)$ form a countable neighborhood basis at $0$ in $X/C$.
Moreover, $X/C$ is Hausdorff because $C$ is closed, and thus
$\bigcap_{n\geq0}B_n=\{0\}$. Consequently,
\[
 \rho(\xi,\eta)=
 \begin{cases}
  0,&\xi=\eta,\\
  2^{-m},&m=\max\{n:\xi-\eta\in B_n\},
 \end{cases}
\]
is a translation-invariant ultrametric inducing the quotient topology.
Thus $X/C$ is metrizable and linearly topologized.

It remains to verify completeness. Let $(\xi_m)_{m\geq0}$ be a Cauchy
sequence in $X/C$. Passing to a subsequence, we may arrange that
\[
  \xi_{m_{n+1}}-\xi_{m_n}\in B_n=q(U_n)
  \qquad(n\geq0).
\]
Choose a representative $x_0\in X$ of $\xi_{m_0}$. For every $n\geq0$,
we may choose a representative $x_{n+1}$ of $\xi_{m_{n+1}}$ such that
$x_{n+1}-x_n\in U_n$. It follows that $(x_n)$ is a Cauchy sequence in
$X$, and hence it converges to some $x\in X$. Therefore,
$(\xi_{m_n})$ converges to $q(x)$. Since $(\xi_m)$ is Cauchy, the whole
sequence also converges to $q(x)$.
\end{proof}

\begin{lemma}\label{lem:KN-embedding}
Let $X$ be a non-discrete complete metrizable linearly topologized $\K$-vector space. Then there exists an injective
$\K$-linear map
\[
  \K^{\N}\longrightarrow X.
\]
Consequently, by the Erd\H{o}s--Kaplansky theorem,
\[   \dim_{\K}(X) \geq \dim_\K(\K^\N) = |\K|^{\aleph_0}. \]
\end{lemma}

\begin{proof}
Since $X$ is metrizable, it has a countable neighborhood basis at $0$, and since $X$ is linearly topologized, each member of this basis contains an open subspace. These open subspaces form a countable neighborhood
basis at $0$, and by taking finite intersections we obtain a decreasing basis
\[   E_0\supseteq E_1\supseteq E_2\supseteq\cdots \]
of clopen subspaces.
Since $X$ is Hausdorff,
\[   \bigcap_{n\geq 0}E_n=\{0\}. \]
Indeed, for every $x\in X\setminus\{0\}$ there is a neighborhood $U$ of
$0$ with $x\notin U$, and there exists $k\geq0$ such that
$E_k\subseteq U$; hence $x\notin E_k$. If the sequence contained only
finitely many distinct subspaces, it would be eventually constant, equal
to $\bigcap_{n\geq0}E_n=\{0\}$. Then $\{0\}$ would be open, contrary to
the assumption that $X$ is non-discrete. Passing to a subsequence and relabelling, we may assume that
\[  E_0\supsetneq E_1\supsetneq E_2\supsetneq\cdots. \]
Choose $e_n\in E_n\setminus E_{n+1}$ for every $n\geq0$.

Let $a=(a_n)_{n\geq0}\in\K^{\N}$, and put
\[  s_m(a):=\sum_{n=0}^{m}a_ne_n.\]
The sequence $(s_m(a))_{m\geq0}$ is Cauchy. To see this, let $U$ be a
neighborhood of $0$ and choose $r$ such that $E_r\subseteq U$.
If $m>n\geq r$, then
\[
  s_m(a)-s_n(a)
  =\sum_{j=n+1}^{m}a_je_j
  \in E_{n+1}\subseteq E_r\subseteq U.
\]
Since $X$ is complete, the sequence $(s_m(a))$ converges. We may define
\[
  \Phi\colon\K^{\N}\longrightarrow X,
  \qquad
  \Phi(a):=\lim_{m\to\infty}s_m(a)
          =\sum_{n\geq0}a_ne_n.
\]
The map $\Phi$ is $\K$-linear by the continuity of both addition and scalar
multiplication.

We now prove that $\Phi$ is injective. Let $a\neq0$, and let $r$ be the
least index such that $a_r\neq0$. Every partial sum of
$\sum_{n>r}a_ne_n$ belongs to $E_{r+1}$. Since $E_{r+1}$ is closed, the
same holds for its limit. On the other hand,
\[
  a_re_r\notin E_{r+1}.
\]
It follows that
\[
  \Phi(a)
  =a_re_r+\sum_{n>r}a_ne_n
  \neq0.
\]
Thus $\Phi$ is injective.
\end{proof}

Following \cite[Definition~14.4.1]{NariciBeckenstein2011}, a linear map
$A\colon X\to Y$ between topological vector spaces is \emph{almost open},
or \emph{nearly open}, if $\cl{A(U)}$ is a neighborhood of $0$ in $Y$
for every neighborhood $U$ of $0$ in $X$.

\begin{proposition}\label{prop:open-mapping}
Let $H$ be countable, let $V$ be a $\K$-vector space with
\[
  \kappa:=\dim_{\K}V < |\K|^{\aleph_0},
\]
and put $X=V^H$. Then every continuous surjective $\K$-linear map
$T\colon X\to X$ is open.
\end{proposition}

\begin{proof}
If $H$ is finite, then $X$ is discrete and there is nothing to prove.
Suppose that $H$ is infinite.

\begin{claim}\label{claim}
For every finite subset $F\subseteq H$, the subspace
$C_F:=\cl{T(U_F)}$ is open in $X$.
\end{claim}
\begin{proof}
The set $C_F$ is a closed subspace of $X$. By \Cref{example} and
\Cref{lem:complete-quotient}, the quotient $X/C_F$ is a complete
metrizable linearly topologized $\K$-vector space. Since
$T(U_F)\subseteq C_F$ and $T$ is surjective, there is a surjective linear
map
\[
  X/U_F\longrightarrow X/C_F,
  \qquad x+U_F\longmapsto T(x)+C_F.
\]
Since $X/U_F\cong V^F$, we obtain
\[
  \dim_{\K}(X/C_F)
  \leq\dim_{\K}(V^F)
  =|F|\cdot\kappa
  <|\K|^{\aleph_0}.
\]
It follows from \Cref{lem:KN-embedding} that $X/C_F$ is discrete, or
equivalently, that $C_F$ is open in $X$.
\end{proof}

Since every neighborhood of $0$ contains some $U_F$, Claim \ref{claim} shows that
$T$ is almost open. We now apply the usual successive-approximation
argument for almost open maps; cf.\ \cite[Section~14.4]
{NariciBeckenstein2011}. Let $O$ be a neighborhood of $0$, and choose
increasing finite subsets
\[
 F_0\subseteq F_1\subseteq\cdots\subseteq H,
 \qquad \bigcup_{n\geq0}F_n=H,
\]
such that $U_{F_0}\subseteq O$, and write $U_n:=U_{F_n}$. Then
$(U_n)_{n\geq0}$ is a decreasing neighborhood basis at $0$ consisting
of clopen subspaces. Put
\[
  W_n:=\cl{T(U_n)}\cap U_n \qquad(n\geq0).
\]
By Claim \ref{claim}, $(W_n)_{n\geq0}$ is also a decreasing neighborhood basis at
$0$ consisting of clopen subspaces. Indeed, each $W_n$ is a neighborhood
of $0$ and $W_n\subseteq U_n$.

Fix $r_0\in W_0$. We shall construct $r_n\in W_n$ and $u_n\in U_n$
inductively. Suppose that $r_n\in W_n$. Since $W_{n+1}$ is an open
subspace, $r_n-W_{n+1}=r_n+W_{n+1}$ is an open neighborhood of $r_n$.
As $r_n\in\cl{T(U_n)}$, this neighborhood meets $T(U_n)$. Hence there is
$u_n\in U_n$ such that $T(u_n)\in r_n-W_{n+1}$. Define
\[
  r_{n+1}:=r_n-T(u_n)\in W_{n+1}.
\]

Let $s_m:=\sum_{n=0}^m u_n$. The sequence $(s_m)$ is Cauchy. In fact, let
$U$ be a neighborhood of $0$ and choose $N$ such that $U_N\subseteq U$.
If $m>\ell\geq N$, then
\[
 s_m-s_\ell=\sum_{n=\ell+1}^m u_n
 \in U_{\ell+1}\subseteq U_N\subseteq U,
\]
because every summand belongs to the subspace $U_{\ell+1}$. Since $X$
is complete, $s_m$ converges to some $u\in X$. Every $s_m$ belongs to
the closed subspace $U_0$, and hence
\[
  u=\sum_{n\geq0}u_n\in U_0\subseteq O.
\]

The recursion gives $T(s_m)=r_0-r_{m+1}$. Moreover,
$r_{m+1}\to0$. Indeed, for every neighborhood $U$ of $0$, choose $N$ such
that $W_N\subseteq U$. Then, for
$m+1\geq N$,
$r_{m+1}\in W_{m+1}\subseteq W_N\subseteq U$. By continuity of $T$,
\[
  T(u)=\lim_{m\to\infty}T(s_m)
      =\lim_{m\to\infty}(r_0-r_{m+1})=r_0.
\]
Since $r_0\in W_0$ was arbitrary, $W_0\subseteq T(O)$. As $W_0$ is a
neighborhood of $0$, so is $T(O)$.

It now follows that $T$ is open. Indeed, let $\Omega\subseteq X$ be open
and let $y\in\Omega$. Choose a neighborhood $O$ of $0$ such that
$y+O\subseteq\Omega$. Then
\[
  T(y)+T(O)=T(y+O)\subseteq T(\Omega),
\]
and $T(y)+T(O)$ is a neighborhood of $T(y)$ because translations are
homeomorphisms. Thus every point of $T(\Omega)$ is interior, and
$T(\Omega)$ is open.
\end{proof}

\begin{theorem}\label{thm:negative}
Let $G$ be any group and let $V$ be a $\K$-vector space satisfying
\[
  \dim_{\K}V<|\K|^{\aleph_0}.
\]
Then every bijective linear cellular automaton $\tau\colon V^G\to V^G$ is reversible.
\end{theorem}

\begin{proof}
Choose a finite memory set $M$ for $\tau$ and let
$H=\langle M\rangle$. The group $H$ is countable. By
\cite[Proposition~1.7.4]{CSCAG2}, the restricted map
$\tau_H\colon V^H\to V^H$ is bijective. It is also linear and continuous.
By \Cref{prop:open-mapping}, $\tau_H$ is open, and hence
$\tau_H^{-1}$ is continuous. Since the inverse is linear and
$H$-equivariant,
\cite[Theorem~8.1.2]{CSCAG2} implies that $\tau_H^{-1}$ is a cellular
automaton. Therefore, $\tau_H$ is reversible. Finally,
\cite[Proposition~1.10.4]{CSCAG2} implies that $\tau$ is reversible.
\end{proof}

\begin{corollary}\label{cor:necessary}
Let $G$ be a group, let $\K$ be a field, and let $V$ be a
$\K$-vector space. If $G$ is locally finite or
$\dim_{\K}V<|\K|^{\aleph_0}$, then every bijective linear cellular
automaton $\tau\colon V^G\to V^G$ is reversible.
\end{corollary}
\begin{proof}
If $\dim_{\K}V<|\K|^{\aleph_0}$, the result follows from
\Cref{thm:negative}. If $G$ is locally finite, it follows from
\cite[Proposition~6.1]{CSC2011}.
\end{proof}

\begin{remark}[A gap in an earlier construction]\label{rem:gap}
Theorem~1.2 of \cite{CSC2011} states that, for every non-periodic group
$G$ and every infinite-dimensional $\K$-vector space $V$, there is a
bijective non-reversible linear cellular automaton $V^G\to V^G$. When
$\dim_{\K}V<|\K|^{\aleph_0}$, this contradicts \Cref{thm:negative}. However, the
proof of the earlier statement contains the following gap.

Let $(v_i)_{i\geq1}$ be a linearly independent sequence in $V$, let $E$
be its span, and, for $j\geq1$, put
\[
 E_j:=\left\langle v_i:
       \frac{(j-1)j}{2}+1\leq i\leq\frac{j(j+1)}{2}\right\rangle.
\]
Then $E=\bigoplus_{j\geq1}E_j$, with $\dim_{\K}E_j=j$, and one may
choose a complement $F$ such that $V=E\oplus F$. In the line preceding
Lemma~5.1 of \cite{CSC2011}, the following identification is used:
\[ V^{\Z}=\left(\bigoplus_{j\geq1}E_j^{\Z}\right)\oplus F^{\Z}. \]
However, only the proper inclusion
\begin{equation}\label{eq:natural-inclusion}
  \bigoplus_{j\geq1}E_j^{\Z}
  \subsetneq
  \left(\bigoplus_{j\geq1}E_j\right)^{\Z}
\end{equation}
holds, since elements on the left space must have finite support, while elements on the right space may have infinite support. In other words, infinite direct sums do not commute with infinite direct products. Consequently, the cellular automaton constructed in \cite[Section~5]{CSC2011} is not surjective. 
\end{remark}


\section{The second implication of \texorpdfstring{\Cref{thm:main}}{the main theorem}}

Throughout this section, $G$ is a non-locally finite group. By K\"onig's
infinity lemma, there exist a finite symmetric subset $S\subseteq G$ and
a sequence $(s_n)_{n\geq1}$ in $S$ such that
\[
  p_0=1_G,
  \qquad
  p_n=s_1s_2\cdots s_n\quad(n\geq1)
\]
are pairwise distinct; see \cite[Lemma~2.4]{YangClosedImage2026}.

We first recall the linear construction of Yang \cite[Sections~3--4]{YangClosedImage2026}. Let
\[
  E=\bigoplus_{n\geq0}\K e_n,
\]
which is isomorphic to $\K[t]$ via the linear map $e_n\mapsto t^n$. Let
$R\colon E\to E$ be the right shift $R(e_n)=e_{n+1}$, and let
$P_0\colon E\to E$ be the projection onto $\K e_0$. For
$s\in S$, define the coordinate projection
$P_s\colon E\to E$ by
\[
  P_s(e_n)=
  \begin{cases}
    e_n,& n\geq1\text{ and }s_n=s,\\
    0,&\text{otherwise}.
  \end{cases}
\]
Define the linear cellular automaton
\[
  \tau_E\colon E^G\longrightarrow E^G
\]
given by
\begin{equation}\label{eq:Yang-automaton}
  \tau_E(x)(g)
  =
  P_0(x(g))-R(x(g))
  +\sum_{s\in S}P_s(x(gs)).
\end{equation}
Its coordinate form, established in
\cite[Lemma~3.2]{YangClosedImage2026}, is
\begin{equation}\label{eq:Yang-coordinate-form}
  (\tau_E(x))_0(g)=x_0(g),
  \qquad
  (\tau_E(x))_n(g)=x_n(gs_n)-x_{n-1}(g)
  \quad(n\geq1).
\end{equation}
Using these identities, Yang proved that the constant configuration
$y\in E^G$ defined by $y(g)=e_0$ for every $g\in G$ belongs to the
closure of $\tau_E(E^G)$, but not to $\tau_E(E^G)$; see
\cite[Propositions~4.1 and~4.2]{YangClosedImage2026}. Indeed, any global
preimage would have to take the value
\[
  e_0+e_1+e_2+\cdots,
\]
at every group element, but this formal sum does not belong to $E$.

Under the isomorphism $E\cong\K[t]$, consider the $t$-adic topology on
$E$. Its completion is the direct product $\K$-vector space
\[
  W=\prod_{n\geq0}\K e_n\cong\K^{\N}\cong\K[[t]].
\]
In particular, the formal sum $e_0+e_1+e_2+\cdots$ is an element of $W$.
We shall show that the corresponding extension of $\tau_E$ is bijective,
but its inverse is not a cellular automaton.

\begin{theorem}\label{thm:positive}
Let $G$ be a non-locally finite group. Then there exists a bijective
non-reversible linear cellular automaton $\tau\colon W^G\to W^G$.
\end{theorem}
\begin{proof}
The endomorphisms $R,P_0$, and $P_s$ of $E$ extend coefficientwise to
$W$. Explicitly, for
$w=\sum_{n\geq0}a_ne_n\in W$, set
\[
 \begin{aligned}
 R(w)&=\sum_{n\geq0}a_ne_{n+1},
 &P_0(w)&=a_0e_0,\\
 P_s(w)&=\sum_{\substack{n\geq1\\s_n=s}}a_ne_n,
 &&s\in S.
 \end{aligned}
\]
These formulas define $\K$-linear endomorphisms of $W$. Therefore,
\eqref{eq:Yang-automaton} defines a linear cellular automaton
\[
  \widehat{\tau}\colon W^G\longrightarrow W^G
\]
with memory set $\{1_G\}\cup S$. The coordinate identities
\eqref{eq:Yang-coordinate-form} remain valid with $\tau_E$ replaced by
$\widehat{\tau}$.

We first prove that $\widehat{\tau}$ is bijective. Let $y\in W^G$.
For every $g\in G$, define the scalars $x_n(g)\in\K$ recursively by
\begin{equation}\label{eq:completed-Yang-inverse}
  x_0(g)=y_0(g),
  \qquad
  x_n(g)=y_n(gs_n^{-1})+x_{n-1}(gs_n^{-1})
  \quad(n\geq1).
\end{equation}
The sequence $(x_n(g))_{n\geq0}$ defines an element
\[
  x(g)=\sum_{n\geq0}x_n(g)e_n\in W.
\]

Substituting \eqref{eq:completed-Yang-inverse} into
\eqref{eq:Yang-coordinate-form} gives
\[
  x_n(gs_n)-x_{n-1}(g)
  =
  y_n(g)
\]
for every $n\geq1$, and $x_0(g)=y_0(g)$. Hence
$\widehat{\tau}(x)=y$. Conversely, any preimage of $y$ must satisfy
\eqref{eq:completed-Yang-inverse}; induction on $n$ therefore determines
all its coefficients uniquely. Thus $\widehat{\tau}$ is bijective.

We now prove that its inverse is not a cellular automaton. Put
\[
  q_n=p_n^{-1}=s_n^{-1}s_{n-1}^{-1}\cdots s_1^{-1}.
\]
The elements $q_n$ are pairwise distinct. For $n\geq1$, let
$z^{(n)}\in W^G$ be the configuration
\[
  z^{(n)}(q_n)=e_0,
  \qquad
  z^{(n)}(g)=0\quad(g\neq q_n).
\]
Write $x=\widehat{\tau}^{-1}(z^{(n)})$. Since all positive
coordinates of $z^{(n)}$ vanish, repeated application of
\eqref{eq:completed-Yang-inverse} gives
\[
  x_n(1_G)
  =
  x_{n-1}(s_n^{-1})
  =
  \cdots
  =
  x_0(s_n^{-1}\cdots s_1^{-1})
  =
  z^{(n)}_0(q_n)
  =
  1.
\]
Therefore
\[
  \widehat{\tau}^{-1}(z^{(n)})(1_G)\neq0.
\]

Suppose that $\widehat{\tau}^{-1}$ has a finite memory set
$N\subseteq G$. Since the $q_n$ are pairwise distinct, there is some
$n$ such that $q_n\notin N$. The configurations
$z^{(n)}$ and $0$ would then agree on $N$, but
\[
  \widehat{\tau}^{-1}(z^{(n)})(1_G)\neq0
  =
  \widehat{\tau}^{-1}(0)(1_G),
\]
which contradicts that  $\widehat{\tau}^{-1}$ has memory set $N$. Thus
$\widehat{\tau}^{-1}$ is not a cellular automaton, and
$\widehat{\tau}$ is not reversible.
\end{proof}

\begin{corollary}\label{cor:sufficient}
Let $G$ be a non-locally finite group and let $V$ be a $\K$-vector space satisfying
\[   \dim_{\K}V\geq|\K|^{\aleph_0}. \]
Then there exists a bijective non-reversible linear cellular automaton
$\tau\colon V^G\to V^G$.
\end{corollary}
\begin{proof}
By the Erd\H{o}s--Kaplansky theorem,
$\dim_{\K}W=\dim_{\K}(\K^{\N})=|\K|^{\aleph_0}$. Hence $V$ contains a
subspace isomorphic to $W$. We identify this subspace with $W$ and choose
a vector-space complement $Z$, so that
\[
  V=W\oplus Z.
\]
Under the induced identification $V^G\cong W^G\times Z^G$, define the
bijective linear cellular automaton
\[
  \tau=\widehat{\tau}\times\id_{Z^G},
\]
where $\widehat{\tau}\colon W^G\to W^G$ is the automaton constructed in
\Cref{thm:positive}. If $\tau^{-1}$ had a finite-memory local rule,
then restricting its input to $W^G\times\{0\}$ and projecting its output
onto $W^G$ would give a finite-memory local rule for
$\widehat{\tau}^{-1}$. This contradicts \Cref{thm:positive}. Hence
$\tau$ is bijective but not reversible.
\end{proof}

\section*{Declaration on the use of artificial intelligence}

The author used OpenAI's ChatGPT 5.6 Sol to assist with mathematical exploration
and exposition. The author wrote and verified all arguments and takes full responsibility for the manuscript.


\end{document}